\documentclass[11pt]{amsart}
\usepackage{amsmath,amssymb,amsthm,mathtools}
\usepackage[margin=1.05in]{geometry}
\usepackage{booktabs}
\usepackage{xcolor}
\usepackage[pdftitle={Characters of F4(C) at involutions},pdfauthor={Chayan Karmakar},colorlinks=true,linkcolor=blue!50!black,citecolor=blue!50!black]{hyperref}

\newcommand{\C}{\mathbb{C}}
\newcommand{\Z}{\mathbb{Z}}
\newcommand{\R}{\mathbb{R}}
\newcommand{\Spin}{\mathrm{Spin}}
\newcommand{\Sp}{\mathrm{Sp}}
\newcommand{\SO}{\mathrm{SO}}
\newcommand{\SL}{\mathrm{SL}}
\newcommand{\GL}{\mathrm{GL}}
\newcommand{\D}{\mathcal{D}}

\theoremstyle{plain}
\newtheorem{theorem}{Theorem}[section]
\newtheorem{proposition}[theorem]{Proposition}
\newtheorem{lemma}[theorem]{Lemma}
\newtheorem{corollary}[theorem]{Corollary}
\theoremstyle{definition}

\theoremstyle{remark}
\newtheorem{remark}[theorem]{Remark}

\title[Characters of $F_4(\C)$ at involutions]{Characters of $F_4(\C)$ at involutions}

\begin{document}
	
	\begin{abstract}
		The exceptional group $F_4(\C)$ has two conjugacy classes of involutions: an involution
		$t$ with centralizer $\Spin(9,\C)$, and the principal involution, with centralizer
		$\bigl(\SL(2,\C)\times\Sp(6,\C)\bigr)/\{\pm1\}$. We determine the values of all
		irreducible characters of $F_4(\C)$ at $t$. For the irreducible representation with
		highest weight $a\omega_1+b\omega_2+c\omega_3+d\omega_4$ (in Bourbaki's labelling), we
		show that the character vanishes at $t$ if and only if $c$ and $d$ are both odd. In all
		other cases its value is $\pm2^{-7}$ times the dimension of an explicit irreducible
		representation of $\Spin(9,\C)$. The proof is a limiting argument in the Weyl character
		formula, combined with a decomposition over the three cosets of $W(B_4)$ in $W(F_4)$ and
		the triality action of $W(F_4)=W(D_4)\rtimes S_3$ on the weight lattice modulo the
		$D_4$-lattice. At the principal involution, Nadimpalli, Pattanayak and Prasad showed that
		the character value is $\pm\tfrac12$ times the dimension of an irreducible representation
		of $\SL(2,\C)\times\Spin(7,\C)$. Comparing the two results, we find that the characters of
		$F_4(\C)$ vanish at the two involutions for exactly the same highest weights. Together
		these results determine all irreducible characters of $F_4(\C)$ at all involutions. We
		also explain how our formula at $t$ relates to the generalized Weyl character formula of
		Gross, Kostant, Ramond and Sternberg.
	\end{abstract}
	
	\author{Chayan Karmakar}
	\email{karmakar.baban28@gmail.com}
	\maketitle
	\tableofcontents

	\section{Introduction}
	
	Let $G$ be a connected reductive group over $\C$ and $\Theta_\lambda$ the character of
	the irreducible representation of $G$ with highest weight $\lambda$. The values of
	$\Theta_\lambda$ at elements of finite order have a long history. A classical theorem of
	Kostant~\cite{Kos} states that at the principal element of order equal to the Coxeter
	number, every irreducible character takes one of the values $0$ and $\pm1$. For the
	general linear group, Littlewood and Richardson~\cite{LR} and, independently,
	Prasad~\cite{Pr} showed that the character of $\GL_{mn}(\C)$ evaluated at elements whose
	eigenvalues are twisted by a primitive $n$-th root of unity either vanishes or
	factorizes into a product of characters of $\GL_m(\C)$. In~\cite{Kar2}, the author gave a
	different proof of Prasad's factorization theorem, based on a decomposition of the Weyl
	group into left cosets of the subgroup that preserves the congruence classes of the
	entries of $\lambda+\rho$.
	
	Ayyer and Kumari~\cite{AK} extended these results to the other classical groups
	$\SO_{2tn+1}$, $\Sp_{2tn}$ and $\mathrm{O}_{2tn}$. They characterized the highest weights
	for which the specialized character is nonzero, and showed that the nonzero characters
	factorize into characters of smaller classical groups; see also~\cite{AK2}. In~\cite{Kar},
	the author studied the characters of the classical groups at elements of order $2$,
	generalizing a special case of these factorization results, and also treated the
	exceptional group $G_2(\C)$. The method there is a deformation of the Weyl character
	formula: one evaluates the character at a nearby regular element and passes to the limit.
	
	Recently, Nadimpalli, Pattanayak and Prasad~\cite{NPP} studied characters of an arbitrary
	connected reductive group at the torsion elements $C_m=\rho^\vee(e^{2\pi i/m})$. They
	proved that $\Theta_\lambda(C_m)\neq0$ if and only if a certain conjugacy condition holds
	in the dual group, and that in this case $\Theta_\lambda(C_m)$ equals, up to sign, the
	ratio of the dimension of an irreducible representation of a subgroup of the dual group to
	an explicit constant $d_m$, whose complete description is due to Polo~\cite{Polo}. For
	$m=2$ this determines the characters at the principal involution $\rho^\vee(-1)$. A
	related nonvanishing criterion appears in the work of Reeder~\cite{Re2}.
	
	An involution of a semisimple group need not be conjugate to the principal one, and the
	results of~\cite{NPP} do not apply to the other classes. The conjugacy classes of
	involutions are classified by Borel and de Siebenthal~\cite{BdS}; see
	also~\cite{Ad,Re1}. The group $G_2(\C)$ has a single class of involutions, which is
	principal, so $F_4(\C)$ is the smallest exceptional group with a non-principal
	involution. It has exactly two conjugacy classes of involutions: the principal involution
	$t_0$, with centralizer $\bigl(\SL(2,\C)\times\Sp(6,\C)\bigr)/\{\pm1\}$, and an involution
	$t$ with centralizer $\Spin(9,\C)$. Both centralizers are connected, since $F_4(\C)$ is
	simply connected~\cite{St}. The purpose of this paper is to compute all irreducible
	characters of $F_4(\C)$ at $t$.
	
	Our main result, Theorem~\ref{thm:main}, states the following. Let
	$\lambda=a\omega_1+b\omega_2+c\omega_3+d\omega_4$ be a dominant weight, in the labelling
	of~\cite{Bou1}. Then $\Theta_\lambda(t)=0$ if and only if $c$ and $d$ are both odd. In
	all other cases,
	\[
	\Theta_\lambda(t)=\pm2^{-7}\dim V^{\Spin(9,\C)}_{\nu}
	\]
	for an explicit dominant weight $\nu$ of $\Spin(9,\C)$, which depends on the parities of
	$c$ and $d$.
	
	The proof follows the deformation method of~\cite{Kar}. Writing the limit of the Weyl
	character formula as a sum over the three cosets of $W(B_4)$ in $W(F_4)$ reduces the
	computation to the sign $\sigma(\nu)=e^{\nu}(t)$ of three weights. The key structural
	input is the triality action of $W(F_4)=W(D_4)\rtimes S_3$ on the quotient $P/D_4$ of the
	weight lattice by the $D_4$-lattice (Lemma~\ref{lem:triality}). This action determines
	the signs and leads to the four cases of Theorem~\ref{thm:main}. We also explain how the
	formula can be deduced from the generalized Weyl character formula of Gross, Kostant,
	Ramond and Sternberg~\cite{GKRS,Ko} for the equal-rank subgroup
	$\Spin(9,\C)\subset F_4(\C)$ (Remark~\ref{rem:GKRS}).
	
	The triplets of $\Spin(9,\C)$-representations attached by~\cite{GKRS} to the irreducible
	representations of $F_4(\C)$ were written down explicitly by Pengpan and
	Ramond~\cite{PR}, who observed that, according to the parities of the Dynkin labels,
	each triplet consists either of one spinorial and two tensorial representations or of
	three spinorial ones. Since $t$ acts by $+1$ on tensorial and by $-1$ on spinorial
	representations of $\Spin(9,\C)$, the values $\Theta_\lambda(t)$ can be read off from
	these results. Our contribution is to state and prove this directly as a result on
	character values, by an argument that uses only the Weyl character formula for
	$F_4(\C)$, and to compare the result with the principal involution.
	
	Comparing with the case $m=2$ of~\cite{NPP} (Theorem~\ref{thm:principal}), we find that
	the characters of $F_4(\C)$ vanish at the two involutions for exactly the same highest
	weights, namely those with $c$ and $d$ both odd; equivalently, by
	Lemma~\ref{lem:classmu}, those with $\lambda+\rho$ in the $D_4$-lattice
	(Corollary~\ref{cor:all-involutions}).
	
	The paper is organized as follows. Section~\ref{sec:main} states the main result. In
	Section~\ref{sec:prelim} we fix notation for the root datum of $F_4(\C)$, the subsystem
	of type $B_4$ and the involution $t$. Section~\ref{sec:triality} studies the action of
	$W(F_4)$ on $P/D_4$, and Section~\ref{sec:anti} collects two lemmas on anti-invariant
	polynomials. Section~\ref{sec:coset} proves the coset formula for $\Theta_\lambda(t)$
	(Proposition~\ref{prop:coset}), and Section~\ref{sec:identity} the vanishing identity
	(Proposition~\ref{prop:identity}). Theorem~\ref{thm:main} is proved in
	Section~\ref{sec:proof}. In Section~\ref{sec:principal} we recall the result
	of~\cite{NPP} at the principal involution and compare the two, and
	Section~\ref{sec:remarks} contains concluding remarks.
	
	\section{Main result}\label{sec:main}
	
	Throughout, $G=F_4(\C)$ and $\lambda=a\omega_1+b\omega_2+c\omega_3+d\omega_4$ is a
	dominant integral weight, with $\Theta_\lambda$ the character of the irreducible
	representation $V_\lambda$ of $G$. Let $t\in G$ be an involution with centralizer $H=C_G(t)\cong\Spin(9,\C)$ (see
	Section~\ref{sec:invo}), and put $\mu=\lambda+\rho$. Throughout, the subscript $H$ refers
	to objects attached to $H$: its root system $\Phi_H$, Weyl group $W_H$, half-sum of
	positive roots $\rho_H=\tfrac12(7,5,3,1)$ (Section~\ref{sec:B4}), and the polynomial
	$\pi_H$ (Section~\ref{sec:anti}). For $y\in\R^4$ define
	\begin{equation}\label{eq:D-def}
		\D(y)=\prod_{i=1}^4\frac{y_i}{(\rho_H)_i}\;\prod_{1\le i<j\le 4}
		\frac{y_i^2-y_j^2}{(\rho_H)_i^2-(\rho_H)_j^2}
		=\frac{1}{226800}\prod_{i=1}^4 y_i\prod_{1\le i<j\le4}\bigl(y_i^2-y_j^2\bigr).
	\end{equation}
	By the Weyl dimension formula, $\D(\nu)=\dim V^{\Spin(9,\C)}_{\nu-\rho_H}$ whenever
	$\nu-\rho_H$ is a dominant weight of $\Spin(9,\C)$. Moreover, $\D$ is anti-invariant
	under $W(B_4)$, that is, $\D(wy)=\varepsilon(w)\,\D(y)$ for all $w\in W(B_4)$ and
	$y\in\R^4$, where $\varepsilon$ denotes the sign character of $W(B_4)$. For example,
	$\D(\lambda+\rho)$ is the dimension of the irreducible representation of $\Spin(9,\C)$
	with highest weight $\lambda+\rho-\rho_H$.
	
	We use Bourbaki's coordinates for $F_4$ (Section~\ref{sec:prelim}), in which
	$\omega_1=(1,1,0,0)$, $\omega_2=(2,1,1,0)$, $\omega_3=\tfrac12(3,1,1,1)$,
	$\omega_4=(1,0,0,0)$ and $\rho=\bigl(\tfrac{11}{2},\tfrac52,\tfrac32,\tfrac12\bigr)$.
	For $\lambda=a\omega_1+b\omega_2+c\omega_3+d\omega_4$ we have
	\[
	\mu=\lambda+\rho=\Bigl(a+2b+\tfrac{3c}{2}+d+\tfrac{11}{2},\ a+b+\tfrac{c}{2}+\tfrac52,\
	b+\tfrac{c}{2}+\tfrac32,\ \tfrac{c}{2}+\tfrac12\Bigr).
	\]
	Let $\beta_1=\tfrac12(1,1,1,1)$ and $\beta_2=\tfrac12(1,1,1,-1)$; these are short
	roots of $F_4$ not lying in the subsystem $\Phi_H$ of type $B_4$
	(Section~\ref{sec:B4}). Since $\langle\beta_i,\beta_i\rangle=1$, the corresponding
	reflections are $s_{\beta_i}(x)=x-2\langle x,\beta_i\rangle\,\beta_i$; explicitly,
	\[
	s_{\beta_1}(x)=x-\tfrac{x_1+x_2+x_3+x_4}{2}\,(1,1,1,1),\qquad
	s_{\beta_2}(x)=x-\tfrac{x_1+x_2+x_3-x_4}{2}\,(1,1,1,-1).
	\]
	
	\begin{theorem}\label{thm:main}
		Let $t\in F_4(\C)$ be an involution with centralizer $\Spin(9,\C)$, let
		$\lambda=a\omega_1+b\omega_2+c\omega_3+d\omega_4$ be a dominant weight, and let
		$\mu=\lambda+\rho$. With $\beta_1,\beta_2$ as above and $\D$ as in \eqref{eq:D-def},
		\[
		\Theta_\lambda(t)=
		\begin{cases}
			0, & c \text{ and } d \text{ both odd},\\[4pt]
			-2^{-7}\,\D(\mu), & c \text{ odd},\ d \text{ even},\\[4pt]
			\phantom{-}2^{-7}\,\D\bigl(s_{\beta_1}\mu\bigr), & c \text{ even},\ d \text{ odd},\\[4pt]
			\phantom{-}2^{-7}\,\D\bigl(s_{\beta_2}\mu\bigr), & c \text{ and } d \text{ both even}.
		\end{cases}
		\]
		In particular, $\Theta_\lambda(t)=0$ if and only if $c$ and $d$ are both odd. Otherwise
		$\Theta_\lambda(t)=\pm2^{-7}\dim V^{\Spin(9,\C)}_{\nu}$, where $\nu=(w\mu)^{+}-\rho_H$,
		with $w=1$, $s_{\beta_1}$ or $s_{\beta_2}$ in the three cases respectively, and
		$(w\mu)^{+}$ denotes the unique $B_4$-dominant element of the $W(B_4)$-orbit of $w\mu$.
	\end{theorem}
	
	\section{Preliminaries}\label{sec:prelim}
	
	\subsection{The root datum of $F_4$}
	We use Bourbaki's conventions \cite[Plate~VIII]{Bou1} in $E=\R^4$ with the standard inner
	product $\langle\,,\rangle$. The simple roots of $F_4(\C)$ are
	\[
	\alpha_1=e_2-e_3,\quad \alpha_2=e_3-e_4,\quad \alpha_3=e_4,\quad
	\alpha_4=\tfrac12(e_1-e_2-e_3-e_4).
	\]
	The root system $\Phi$ consists of the $24$ long roots $\pm e_i\pm e_j$ ($i<j$) and the
	$24$ short roots $\pm e_i$ and $\tfrac12(\pm e_1\pm e_2\pm e_3\pm e_4)$. The positive
	roots $\Phi^+$ are $e_i$, $e_i\pm e_j$ ($i<j$) and $\tfrac12(e_1\pm e_2\pm e_3\pm e_4)$,
	so $|\Phi^+|=24$. Let $W=W(F_4)$ be the Weyl group; recall that $|W|=1152$. We write
	$\varepsilon(w)=\det(w)$ for the sign character of $W$; its restriction to any reflection
	subgroup is the sign character of that subgroup. The fundamental weights and $\rho$ are
	\begin{equation}\label{eq:f4-weights}
		\omega_1=(1,1,0,0),\quad \omega_2=(2,1,1,0),\quad
		\omega_3=\tfrac12(3,1,1,1),\quad \omega_4=(1,0,0,0),\quad
		\rho=\tfrac12(11,5,3,1).
	\end{equation}
	The group $G=F_4(\C)$ is both simply connected and adjoint, and its weight lattice is
	\[
	P=\Z^4\cup\bigl(\Z^4+\tfrac12(1,1,1,1)\bigr),
	\]
	the set of vectors whose coordinates are either all integers or all in $\Z+\tfrac12$.
	
	Let $T\subset G$ be a maximal torus with Lie algebra $\mathfrak t$. We identify
	$\mathfrak t$ with $\C^4=\C\otimes_\R\R^4$ in such a way that each coroot
	$\alpha^\vee\in\mathfrak t$ corresponds to the vector $2\alpha/\langle\alpha,\alpha\rangle$,
	and we extend $\langle\,,\rangle$ $\C$-bilinearly. A weight $\lambda\in P$ is then the
	linear form $h\mapsto\langle\lambda,h\rangle$ on $\mathfrak t$. The coroots are
	$\pm e_i\pm e_j$, $\pm2e_i$ and $(\pm1,\pm1,\pm1,\pm1)$, so the coroot lattice is
	\[
	Q^\vee=D_4=\Bigl\{x\in\Z^4:\ \textstyle\sum_i x_i\in2\Z\Bigr\},
	\]
	which is also the lattice spanned by the long roots. Throughout, $D_4$ denotes this
	lattice. Since $G$ is simply connected, the exponential map $\exp\colon\mathfrak t\to T$
	is a surjective homomorphism with kernel $2\pi i\,Q^\vee$. Each $\lambda\in P$ defines a
	character $e^\lambda$ of $T$ by
	\begin{equation}\label{eq:exp-char}
		e^{\lambda}(\exp h)=e^{\langle\lambda,h\rangle},\qquad h\in\mathfrak t,
	\end{equation}
	which is well defined because $\langle\lambda,\alpha^\vee\rangle\in\Z$ for all roots
	$\alpha$. For a root $\alpha$ we also write $\alpha(x)$ for $e^\alpha(x)$; thus
	$\alpha(\exp h)=e^{\langle\alpha,h\rangle}$.
	
	\subsection{The subsystem $B_4$}\label{sec:B4}
	The roots $\pm e_i\pm e_j$ and $\pm e_i$ form a subsystem $\Phi_H\subset\Phi$ of type
	$B_4$, which is the root system of $H$ (Section~\ref{sec:invo}), with positive system
	\[
	\Phi_H^+=\{e_i\pm e_j\ (i<j),\ e_i\},\qquad |\Phi_H^+|=16.
	\]
	The half-sum of its positive roots is $\rho_H=\tfrac12(7,5,3,1)$. Its Weyl group
	$W_H=W(B_4)$ is the group of signed permutations of the coordinates, of order
	$|W_H|=2^4\cdot4!=384$, so $|W:W_H|=3$. The open dominant chamber of $\Phi_H$ is
	$\{\nu\in\R^4:\ \nu_1>\nu_2>\nu_3>\nu_4>0\}$. The remaining $8$ positive roots of $\Phi$
	are the vectors $\tfrac12(e_1\pm e_2\pm e_3\pm e_4)$.
	
	\subsection{The involution $t$}\label{sec:invo}
	Define $\sigma\colon P\to\{\pm1\}$ by $\sigma(\nu)=(-1)^{2\nu_1}$, where $\nu_1$ is the
	first coordinate of $\nu$; this is well defined since $2\nu_1\in\Z$ for $\nu\in P$, and it
	is a group homomorphism. Explicitly, $\sigma(\nu)=1$ if $\nu\in\Z^4$ and $\sigma(\nu)=-1$
	if $\nu\in(\Z+\tfrac12)^4$. Since $X^*(T)=P$ and $T\cong\operatorname{Hom}(X^*(T),\C^*)$,
	there is a unique element $t\in T$ with
	\[
	e^{\nu}(t)=\sigma(\nu)\qquad\text{for all }\nu\in P;
	\]
	in the notation of \eqref{eq:exp-char}, $t=\exp(2\pi i\,e_1)$. Since $\sigma^2=1$ and
	$\sigma\neq1$, the element $t$ is an involution. Moreover, $\alpha(t)=1$ exactly for the
	integral roots, that is, for $\alpha\in\Phi_H$.
	
	Recall that $H=C_G(t)$. Since $G$ is simply connected, $H$ is connected~\cite{St}, so it is a
	connected reductive group with maximal torus $T$ and root system $\Phi_H$ of type $B_4$.
	Its character lattice $X^*(T)=P$ is the full weight lattice of $\Phi_H$: the coroots of
	$\Phi_H$ are $\pm e_i\pm e_j$ and $\pm2e_i$, which span $D_4$, and the dual lattice of
	$D_4$ is $P$. Hence $H$ is simply connected, and $H\cong\Spin(9,\C)$. This is the
	maximal-rank subgroup $\Spin(9,\C)\subset F_4(\C)$ of Borel and de
	Siebenthal~\cite{BdS}; see also~\cite{Ad}.

	\section{The action of $W(F_4)$ on $P/D_4$}\label{sec:triality}
	
	Recall that $D_4=\{x\in\Z^4:\sum_i x_i\ \text{even}\}$ is the lattice spanned by the long
	roots. Then $P/D_4\cong(\Z/2)^2$, with the four classes
	\begin{equation}\label{eq:D4-classes}
		\begin{aligned}
			0&=D_4, & V&=\{x\in\Z^4:\textstyle\sum x_i\ \text{odd}\},\\
			S&=\{x\in(\Z+\tfrac12)^4:\textstyle\sum x_i\ \text{odd}\}, &
			S'&=\{x\in(\Z+\tfrac12)^4:\textstyle\sum x_i\ \text{even}\}.
		\end{aligned}
	\end{equation}
	(For $x\in(\Z+\frac12)^4$ the sum $\sum x_i$ is an integer.) The group law satisfies
	$V+S=S'$, and every nonzero class has order $2$. By definition, $\sigma=1$ on $D_4\cup V$
	and $\sigma=-1$ on $S\cup S'$.
	
	\begin{lemma}\label{lem:triality}
		\leavevmode
		\begin{enumerate}
			\item $W$ preserves $D_4$, hence acts on $P/D_4$ fixing $0$.
			\item Reflections in long roots act trivially on $P/D_4$; hence $W(D_4)$ acts trivially.
			\item $W$ permutes $\{V,S,S'\}$ as the full symmetric group $S_3$, and $W_H$ is exactly
			the stabilizer of $V$.
			\item $W_H$ fixes $V$ and interchanges $S$ and $S'$. In particular $\sigma(v\nu)=\sigma(\nu)$
			for $v\in W_H$ and $\nu\in P$.
		\end{enumerate}
	\end{lemma}
	
	\begin{proof}
		(1) Since $W$ permutes the roots and preserves their lengths, the set of long roots is
		$W$-stable, and therefore $W$ preserves the lattice $D_4$ they generate. Hence $W$ acts
		$\Z$-linearly on $P/D_4$ and fixes the zero class.
		
		(2) Let $\alpha$ be a long root. For $\nu\in P$ we have
		\[
		s_\alpha\nu=\nu-\langle\nu,\alpha^\vee\rangle\,\alpha,
		\qquad \langle\nu,\alpha^\vee\rangle\in\Z,
		\]
		and $\alpha\in D_4$, so $s_\alpha\nu\equiv\nu\pmod{D_4}$. Thus every reflection in a long
		root acts trivially on $P/D_4$, and as $W(D_4)$ is generated by such reflections, $W(D_4)$
		acts trivially.
		
		Because $W$ acts on $P/D_4\cong(\Z/2)^2$ by group automorphisms, and
		$\operatorname{Aut}\big((\Z/2)^2\big)\cong S_3$ acts faithfully on the three nonzero
		elements, we obtain a homomorphism $\tau\colon W\to S_3=\operatorname{Sym}\{V,S,S'\}$.
		By (2), $\tau$ is determined by its values on the reflections in short roots, so it
		suffices to examine these.
		
		(4) The group $W_H$ is generated by the reflections in the long roots $e_i\pm e_j$
		together with the reflections in the short roots $e_i$. The former act trivially on
		$P/D_4$ by (2). For the latter,
		\[
		s_{e_i}\nu=\nu-2\nu_i e_i,
		\]
		which negates the $i$-th coordinate; this preserves $\Z^4$ and its half-integral coset and
		alters $\sum_k\nu_k$ by $2\nu_i$. If $\nu$ is integral then $2\nu_i$ is even, so $s_{e_i}$
		fixes the class $V$; if $\nu$ is half-integral then $2\nu_i$ is odd, so $s_{e_i}$ reverses
		the parity of $\sum_k\nu_k$ and hence interchanges $S$ and $S'$. Therefore $W_H$ fixes $V$
		and interchanges $S$ and $S'$. In particular $\sigma$ is constant on the $W_H$-orbits of
		classes, whence $\sigma(v\nu)=\sigma(\nu)$ for all $v\in W_H$ and $\nu\in P$.
		
		(3) It remains to show that $W$ does not stabilize $V$. For the short root
		$\alpha_4=\tfrac12(e_1-e_2-e_3-e_4)$ we have $\langle\nu,\alpha_4^\vee\rangle
		=\nu_1-\nu_2-\nu_3-\nu_4$ and
		\[
		s_{\alpha_4}\nu=\nu-(\nu_1-\nu_2-\nu_3-\nu_4)\,\alpha_4.
		\]
		If $\nu\in V$ is integral with odd coordinate sum, then
		$\nu_1-\nu_2-\nu_3-\nu_4\equiv\sum_k\nu_k\equiv1\pmod 2$, so $s_{\alpha_4}\nu$ is
		half-integral; that is, $s_{\alpha_4}$ carries $V$ into $\{S,S'\}$. Hence $\tau(W)$
		contains the transposition $(S\,S')$ together with an element not fixing $V$, so
		$\tau(W)=S_3$ and $W$ permutes $\{V,S,S'\}$ as the full symmetric group.
		
		Finally, the $W$-orbit of $V$ has size $3$, so $\operatorname{Stab}_W(V)$ has order
		$|W|/3=1152/3=384=|W_H|$. As $W_H$ fixes $V$, it is contained in this stabilizer, and
		equality of orders forces $W_H=\operatorname{Stab}_W(V)$.
	\end{proof}
	
	\begin{remark}\label{rem:cosets}
		This is the concrete form of the decomposition $W(F_4)=W(D_4)\rtimes S_3$: the normal
		subgroup $W(D_4)$ acts trivially on $P/D_4$, and the quotient $S_3$ is the triality group
		permuting the vector class $V$ and the two half-spin classes $S$ and $S'$. Since
		$W_H$ is the stabilizer of $V$ (Lemma~\ref{lem:triality}(3)), the map
		$W_Hw\mapsto w^{-1}V$ is a bijection from the three right cosets of $W_H$ in $W$ onto
		$\{V,S,S'\}$.
	\end{remark}
	
	\begin{lemma}\label{lem:classmu}
		For a dominant weight $\lambda=a\omega_1+b\omega_2+c\omega_3+d\omega_4$, the class of
		$\mu=\lambda+\rho$ in $P/D_4$ is $cS+dV+S'$; explicitly, $\mu$ lies in
		$0,\,V,\,S,\,S'$ according as $(c,d)\equiv(1,1),\,(1,0),\,(0,1),\,(0,0)\pmod 2$.
	\end{lemma}
	
	\begin{proof}
		By \eqref{eq:f4-weights}, $\omega_1,\omega_2\in D_4$, while $\omega_3\in S$, $\omega_4\in V$
		and $\rho\in S'$. Hence $\mu\equiv c\omega_3+d\omega_4+\rho\equiv cS+dV+S'\pmod{D_4}$, and the
		four cases follow from the group law of $P/D_4\cong(\Z/2)^2$ (using $S+S'=V$).
	\end{proof}
	
	\section{Anti-invariant polynomials}\label{sec:anti}
	
	Let $R$ be a root system in a Euclidean space $E$ with inner product
	$\langle\,,\rangle$, positive system $R^+$, Weyl group $W_R$ and sign character
	$\varepsilon$. Set $N=|R^+|$, let $\rho_R=\tfrac12\sum_{\alpha\in R^+}\alpha$ be half the
	sum of the positive roots, and define $\pi_R(h)=\prod_{\alpha\in R^+}\langle\alpha,h\rangle$
	for $h\in E$.
	
	\begin{lemma}\label{lem:anti}
		Every $W_R$-anti-invariant polynomial on $E$ is divisible by $\pi_R$. In
		particular, a nonzero $W_R$-anti-invariant polynomial has degree at least $N$, so
		every $W_R$-anti-invariant polynomial of degree ${}<N$ vanishes.
	\end{lemma}
	
	This is standard (see the discussion of skew-invariant polynomials in
	\cite[Ch.~3]{Hu2}): each linear form $\langle\alpha,\,\cdot\,\rangle$ divides every
	polynomial that changes sign under $s_\alpha$, and the forms $\langle\alpha,\,\cdot\,\rangle$
	for $\alpha\in R^+$ are pairwise coprime, so their product $\pi_R$ divides every
	$W_R$-anti-invariant polynomial.
	
	\begin{lemma}\label{lem:lowest}
		For $\nu,h\in E$,
		\[
		\frac{1}{N!}\sum_{v\in W_R}\varepsilon(v)\,\langle v\nu,h\rangle^{N}
		=\frac{\pi_R(\nu)}{\pi_R(\rho_R)}\,\pi_R(h).
		\]
		In particular, if $\nu=\nu'+\rho_R$ with $\nu'$ dominant integral, then
		$\pi_R(\nu)/\pi_R(\rho_R)=\dim V_{\nu'}$, where $V_{\nu'}$ is the irreducible
		representation of highest weight $\nu'$ of the complex semisimple Lie algebra with root
		system $R$.
	\end{lemma}
	
	\begin{proof}
		Write $a(\nu,h)=\frac{1}{N!}\sum_{v\in W_R}\varepsilon(v)\langle v\nu,h\rangle^{N}$, a
		polynomial in $(\nu,h)$, homogeneous of degree $N$ in each variable. For $w\in W_R$,
		substituting $u=w^{-1}v$ and using $\langle v\nu,wh\rangle=\langle w^{-1}v\nu,h\rangle$, and
		substituting $u=vw$ respectively, we get
		\[
		a(\nu,wh)=\varepsilon(w)\,a(\nu,h)
		\qquad\text{and}\qquad
		a(w\nu,h)=\varepsilon(w)\,a(\nu,h).
		\]
		Since $a(\nu,\cdot)$ is $W_R$-anti-invariant of degree $N=\deg\pi_R$, Lemma~\ref{lem:anti}
		gives $a(\nu,h)=c(\nu)\,\pi_R(h)$, where $c(\nu)$ does not depend on $h$. Fixing $h_0$
		with $\pi_R(h_0)\neq0$, we have $c(\nu)=a(\nu,h_0)/\pi_R(h_0)$, so $c$ is a
		$W_R$-anti-invariant polynomial of degree $N$, and by Lemma~\ref{lem:anti} again
		$c(\nu)=K\pi_R(\nu)$ for a constant $K$. Thus
		\[
		a(\nu,h)=K\,\pi_R(\nu)\,\pi_R(h).
		\]
		To determine $K$, put $\nu=\rho_R$. By the Weyl denominator identity,
		\[
		\sum_{v\in W_R}\varepsilon(v)\,e^{\langle v\rho_R,h\rangle}
		=\prod_{\alpha\in R^+}\bigl(e^{\langle\alpha,h\rangle/2}-e^{-\langle\alpha,h\rangle/2}\bigr).
		\]
		The homogeneous component of degree $N$ of the left side is $a(\rho_R,h)$. Each factor on
		the right is $\langle\alpha,h\rangle$ plus terms of higher degree, so the right side has no
		terms of degree below $N$, and its component of degree $N$ is $\pi_R(h)$. Hence
		$K\pi_R(\rho_R)=1$. Since $\langle\alpha,\rho_R\rangle>0$ for all $\alpha\in R^+$, we have
		$\pi_R(\rho_R)\neq0$, so $K=1/\pi_R(\rho_R)$. The last assertion is the Weyl dimension
		formula.
	\end{proof}
	
	We apply the results of this section with $R=\Phi_H$, so that $W_R=W_H$, $N=16$,
	$\rho_R=\rho_H$ and $\pi_R=\pi_H$, where
	\[
	\pi_H(\nu)=\prod_{\alpha\in\Phi_H^+}\langle\alpha,\nu\rangle
	=\prod_{i}\nu_i\prod_{i<j}\bigl(\nu_i^2-\nu_j^2\bigr),\qquad \nu\in E.
	\]
	In this notation the polynomial $\D$ of \eqref{eq:D-def} is
	\[
	\D(\nu)=\frac{\pi_H(\nu)}{\pi_H(\rho_H)},\qquad \pi_H(\rho_H)=226800.
	\]
	Thus $\D$ is a $W_H$-anti-invariant polynomial on $E$. If $\nu\in P$ is regular and
	$B_4$-dominant, then $\nu-\rho_H$ is a dominant weight of $\Spin(9,\C)$, and
	$\D(\nu)=\dim V^{\Spin(9,\C)}_{\nu-\rho_H}$ by the Weyl dimension formula
	(Lemma~\ref{lem:lowest}). In general, $\D(\nu)=\varepsilon(u)\,\D(\nu^+)$, where $\nu^+$
	is the $B_4$-dominant element of the $W_H$-orbit of $\nu$ and $u\in W_H$ satisfies
	$u\nu=\nu^+$.
	
	\section{A coset formula for $\Theta_\lambda(t)$}\label{sec:coset}
	Recall that $W=W(F_4)$ and $W_H=W(B_4)$, that $|W:W_H|=3$, and let $\sigma$ be the
	homomorphism of Section~\ref{sec:invo}, so that $e^\nu(t)=\sigma(\nu)$ for $\nu\in P$.
	Let $\mu\in P$ be regular; we shall apply the results below to $\mu=\lambda+\rho$. Then
	$\mu$ is also regular for $\Phi_H$. Since $W_H$ acts simply transitively on the
	$B_4$-chambers, each right coset of $W_H$ in $W$ contains a unique element $w_i$ such that
	$w_i\mu$ lies in the dominant $B_4$-chamber. We write
	\[
	W=W_Hw_1\sqcup W_Hw_2\sqcup W_Hw_3
	\]
	accordingly, and call $w_1,w_2,w_3$ the coset representatives attached to $\mu$. The next
	proposition expresses the value of $\Theta_\lambda$ at $t$ in terms of these
	representatives, with $\D$ as in \eqref{eq:D-def}.
	
	\begin{proposition}\label{prop:coset}
		For $\mu=\lambda+\rho$,
		\[
		\Theta_\lambda(t)=-\frac{1}{2^{8}}\sum_{i=1}^{3}\varepsilon(w_i)\,\sigma(w_i\mu)\,\D(w_i\mu).
		\]
	\end{proposition}
	
	\begin{proof}
		Let $h\in\R^4$ with $\pi_H(h)\neq0$, and put $x_s=t\exp(sh)$ for real $s$. By the
		definition of $\sigma$ in Section~\ref{sec:invo} and by \eqref{eq:exp-char},
		\[
		e^{\nu}(x_s)=e^{\nu}(t)\,e^{\nu}\bigl(\exp(sh)\bigr)=\sigma(\nu)\,e^{s\langle\nu,h\rangle},
		\qquad \nu\in P.
		\]
		The roots of $\Phi_H$ are the integral roots $\pm e_i\pm e_j$ and $\pm e_i$, so
		$\sigma(\alpha)=1$ for $\alpha\in\Phi_H$. The remaining roots are half-integral, so
		$\sigma(\alpha)=-1$ for $\alpha\in\Phi\setminus\Phi_H$. Hence
		$\alpha(x_s)=e^{s\langle\alpha,h\rangle}\neq1$ for $\alpha\in\Phi_H$ and $s\neq0$, since
		$\langle\alpha,h\rangle\neq0$; and $\alpha(x_s)=-e^{s\langle\alpha,h\rangle}\neq1$ for
		$\alpha\notin\Phi_H$. Thus $x_s$ is regular for $s\neq0$. By the Weyl character
		formula~\cite[\S24]{Hu1}, $\Theta_\lambda(x_s)=A_\mu(x_s)/A_\rho(x_s)$, where
		$A_\nu=\sum_{w\in W}\varepsilon(w)e^{w\nu}$, and since $\Theta_\lambda$ is continuous,
		\[
		\Theta_\lambda(t)=\lim_{s\to0}\frac{A_\mu(x_s)}{A_\rho(x_s)}.
		\]
		
		\smallskip\noindent\emph{The Weyl denominator.} We have
		\begin{equation}\label{eq:Arho-xs}
			A_\rho(x_s)=e^{\rho}(x_s)\prod_{\alpha\in\Phi^+}\bigl(1-e^{-\alpha}(x_s)\bigr).
		\end{equation}
		Since $\rho\in S'$ by \eqref{eq:D4-classes}, we have $e^{\rho}(x_s)=-e^{s\langle\rho,h\rangle}$.
		Together with the values of $\alpha(x_s)$ computed above, this gives
		\begin{equation}\label{eq:Arho-xs-expanded}
			A_{\rho}(x_s)=-e^{s\langle\rho,h\rangle}
			\prod_{\alpha\in\Phi^{+}_{H}}\bigl(1-e^{-s\langle\alpha,h\rangle}\bigr)
			\prod_{\alpha\in\Phi^{+}\setminus\Phi^{+}_{H}}\bigl(1+e^{-s\langle\alpha,h\rangle}\bigr).
		\end{equation}
		Since $|\Phi^+|=24$ and $|\Phi_H^+|=16$, the second product has $8$ factors. Expanding
		each factor in $s$, we obtain
		\begin{equation}\label{eq:Arho-leading}
			\begin{aligned}
				A_\rho(x_s)&=-e^{s\langle\rho,h\rangle}
				\prod_{\alpha\in\Phi_H^+}\bigl(s\langle\alpha,h\rangle+O(s^2)\bigr)
				\prod_{\alpha\in\Phi^+\setminus\Phi_H^+}\bigl(2-s\langle\alpha,h\rangle+O(s^2)\bigr)\\
				&=-2^{8}\,s^{16}\,\pi_H(h)+O(s^{17}).
			\end{aligned}
		\end{equation}
		
		\smallskip\noindent\emph{The Weyl numerator.} Expanding the exponentials,
		\begin{equation}\label{eq:WN}
			A_{\mu}(x_s)=\sum_{w\in W}\varepsilon(w)\,\sigma(w\mu)\,e^{s\langle w\mu,h\rangle}
			=\sum_{k\ge0}\frac{s^k}{k!}\,p_k(h),
		\end{equation}
		where $p_k(h)=\sum_{w\in W}\varepsilon(w)\sigma(w\mu)\langle w\mu,h\rangle^k$. For
		$v\in W_H$ we have
		\begin{equation}\label{eq:pk-key1}
			\begin{aligned}
				p_k(vh)&=\sum_{w\in W}\varepsilon(w)\,\sigma(w\mu)\,\langle w\mu,vh\rangle^k\\
				&=\sum_{w\in W}\varepsilon(w)\,\sigma(w\mu)\,\langle v^{-1}w\mu,h\rangle^k\\
				&=\varepsilon(v)\sum_{w\in W}\varepsilon(v^{-1}w)\,\sigma(v^{-1}w\mu)\,
				\langle v^{-1}w\mu,h\rangle^k\\
				&=\varepsilon(v)\,p_k(h).
			\end{aligned}
		\end{equation}
		The second equality uses the $W$-invariance of $\langle\,,\rangle$. The third uses
		$\varepsilon(w)=\varepsilon(v)\varepsilon(v^{-1}w)$ and $\sigma(w\mu)=\sigma(v^{-1}w\mu)$,
		which holds by Lemma~\ref{lem:triality}(4) since $v^{-1}\in W_H$. The last follows by
		substituting $u=v^{-1}w$, which runs over $W$ as $w$ does. Thus each $p_k$ is
		$W_H$-anti-invariant. By Lemma~\ref{lem:anti}, applied with $R=\Phi_H$ and $N=16$, we get
		$p_k=0$ for $k<16$, and \eqref{eq:WN} becomes
		\begin{equation}\label{eq:Amu-leading}
			A_\mu(x_s)=\frac{s^{16}}{16!}\,p_{16}(h)+O(s^{17}).
		\end{equation}
		Since $\pi_H(h)\neq0$, \eqref{eq:Arho-leading} and \eqref{eq:Amu-leading} give
		\begin{equation}\label{eq:Theta-limit}
			\Theta_\lambda(t)=-\frac{p_{16}(h)}{16!\cdot2^{8}\,\pi_H(h)}.
		\end{equation}
		
		\smallskip\noindent\emph{Coset decomposition.} Write $W=W_Hw_1\sqcup W_Hw_2\sqcup W_Hw_3$
		with the $w_i$ chosen as above. Using $\varepsilon(ww_i)=\varepsilon(w)\varepsilon(w_i)$,
		the identity $\sigma(ww_i\mu)=\sigma(w_i\mu)$ for $w\in W_H$ (Lemma~\ref{lem:triality}(4)),
		and Lemma~\ref{lem:lowest} with $R=\Phi_H$, we obtain
		\begin{equation}\label{eq:p16}
			\begin{aligned}
				p_{16}(h)&=\sum_{i=1}^{3}\varepsilon(w_i)\,\sigma(w_i\mu)
				\sum_{w\in W_H}\varepsilon(w)\,\langle ww_i\mu,h\rangle^{16}\\
				&=16!\,\pi_H(h)\sum_{i=1}^{3}\varepsilon(w_i)\,\sigma(w_i\mu)\,\D(w_i\mu).
			\end{aligned}
		\end{equation}
		Substituting \eqref{eq:p16} into \eqref{eq:Theta-limit} proves the proposition.
	\end{proof}
	
	\begin{remark}\label{rem:GKRS}
		Proposition~\ref{prop:coset} can also be deduced from the generalized Weyl character
		formula of Gross, Kostant, Ramond and Sternberg~\cite{GKRS}; see also Kostant~\cite{Ko}.
		Let $\mathfrak g$ be a complex semisimple Lie algebra and $\mathfrak h\subset\mathfrak g$ a
		reductive subalgebra of full rank. For every dominant weight $\lambda$ they prove the
		identity of virtual $\mathfrak h$-modules
		\[
		V_\lambda\otimes(S^+-S^-)=\sum_{c\in C}\varepsilon(c)\,U_{c(\lambda+\rho)-\rho_H}.
		\]
		Here $S^\pm$ are the restrictions to $\mathfrak h$ of the two half-spin representations of
		$\mathfrak{so}(\mathfrak g/\mathfrak h)$, normalized so that $S^+$ contains the weight
		$\rho-\rho_H$. The set $C\subset W$ consists of the elements mapping the dominant chamber of
		$\mathfrak g$ into that of $\mathfrak h$; it contains exactly one element of each coset in
		$W_H\backslash W$. Finally, $U_\nu$ denotes the irreducible $\mathfrak h$-module of highest
		weight $\nu$.
		
		In our situation $\mathfrak g$ and $\mathfrak h$ are the Lie algebras of $G$ and
		$H\cong\Spin(9,\C)$, $\dim\mathfrak g/\mathfrak h=16$, and $C=\{w_1,w_2,w_3\}$ are the
		elements of Proposition~\ref{prop:coset}. Since $\rho-\rho_H=(2,0,0,0)\in P$, the modules
		$S^\pm$ are representations of $H$, so the identity holds for characters of $H$ and may be
		evaluated at $t$. The character of $S^+-S^-$ is
		$e^{\rho-\rho_H}\prod_{\alpha\in\Phi^+\setminus\Phi_H^+}(1-e^{-\alpha})$, whose value at
		$t$ is $2^{8}$. Since $t$ is central in $H$, it acts on $U_{w_i\mu-\rho_H}$ by the scalar
		$e^{w_i\mu-\rho_H}(t)=-\sigma(w_i\mu)$, and by the Weyl dimension formula for $H$ this
		module has dimension $\D(w_i\mu)$. Comparing both sides gives
		Proposition~\ref{prop:coset}. The three highest weights $w_i\mu-\rho_H$ were written down
		explicitly in Dynkin labels by Pengpan and Ramond~\cite[(62)]{PR}, who also determined,
		in terms of the parities of the labels, which of them are spinorial~\cite[\S2]{PR}.
		
		More generally, in a simply connected group the centralizer of a semisimple element is
		connected, of full rank, and contains that element in its centre. An analogous evaluation
		therefore applies to other elements of finite order; Proposition~\ref{prop:coset} is the
		case of the involution $t$. Our proof does not use~\cite{GKRS}: it proceeds directly from
		the Weyl character formula for $G$ by the limiting argument above.
	\end{remark}
	
	\section{A vanishing identity}\label{sec:identity}
	
	\begin{proposition}\label{prop:identity}
		Let $\mu\in P$ be regular, and let $w_1,w_2,w_3$ be the coset representatives attached to
		$\mu$ as above. Then
		\[
		\sum_{i=1}^{3}\varepsilon(w_i)\,\D(w_i\mu)=0.
		\]
	\end{proposition}
	
	\begin{proof}
		The polynomial $q(h)=\sum_{w\in W}\varepsilon(w)\langle w\mu,h\rangle^{16}$ is
		$W$-anti-invariant of degree $16<24=|\Phi^+|$, so $q=0$ by Lemma~\ref{lem:anti},
		applied with $R=\Phi$. On the other hand, the computation \eqref{eq:p16}, with $\sigma$
		replaced by $1$, gives
		\[
		q(h)=16!\,\pi_H(h)\sum_{i=1}^{3}\varepsilon(w_i)\,\D(w_i\mu).
		\]
		Choosing $h$ with $\pi_H(h)\neq0$, we conclude that the sum vanishes.
	\end{proof}
	
	\begin{remark}\label{rem:identity-GKRS}
		Proposition~\ref{prop:identity} is consistent with the formula of~\cite{GKRS} recalled in
		Remark~\ref{rem:GKRS}. Taking dimensions in that identity of virtual $H$-modules, the left
		side $V_\lambda\otimes(S^+-S^-)$ has dimension $0$, since
		$\dim S^+=\dim S^-=2^{7}$. By the Weyl dimension formula for $H$, the right side has
		dimension $\sum_{i=1}^{3}\varepsilon(w_i)\,\D(w_i\mu)$. This gives
		Proposition~\ref{prop:identity} for $\mu=\lambda+\rho$ with $\lambda$ dominant. The general
		case reduces to this one, because every regular $\mu\in P$ is $W$-conjugate to such a
		weight, and replacing $\mu$ by $w\mu$ multiplies the sum by $\varepsilon(w)$.
	\end{remark}
	
	\section{Proof of Theorem~\ref{thm:main}}\label{sec:proof}
	
	\begin{proof}[Proof of Theorem~\ref{thm:main}]
		Throughout, $\mu=\lambda+\rho$, which is regular and dominant for $F_4$. Recall that
		$\sigma=1$ on $D_4\cup V$ and $\sigma=-1$ on $S\cup S'$.
		
		\smallskip\noindent\emph{Choice of coset representatives.}
		For $u\in W_H$ and $w\in W$ we have $\varepsilon(uw)=\varepsilon(u)\varepsilon(w)$,
		$\sigma(uw\mu)=\sigma(w\mu)$ by Lemma~\ref{lem:triality}(4), and
		$\D(uw\mu)=\varepsilon(u)\D(w\mu)$ because $\pi_H$ is $W_H$-anti-invariant. Hence the
		quantities $\varepsilon(w)\D(w\mu)$ and $\varepsilon(w)\sigma(w\mu)\D(w\mu)$ depend only on
		the coset $W_Hw$. Consequently, Propositions~\ref{prop:coset} and~\ref{prop:identity} hold
		for any set of representatives $w_1,w_2,w_3$ of $W_H\backslash W$.
		
		By Remark~\ref{rem:cosets}, the map $W_Hw\mapsto w^{-1}V$ is a bijection from
		$W_H\backslash W$ onto $\{V,S,S'\}$. Now suppose $\mu$ lies in a nonzero class $X$. Then
		$w\mu\in V$ if and only if $w^{-1}V=X$, so exactly one coset maps $\mu$ into $V$. The
		other two map $\mu$ into the two nonzero classes different from $V$, on which $\sigma=-1$.
		
		\smallskip\noindent\textbf{Case A: $c$ and $d$ both odd.}
		By Lemma~\ref{lem:classmu}, $\mu\in D_4$. Since $W$ preserves $D_4$
		(Lemma~\ref{lem:triality}(1)), we have $\sigma(w_i\mu)=1$ for all $i$.
		Propositions~\ref{prop:coset} and~\ref{prop:identity} then give
		\[
		\Theta_\lambda(t)=-\frac{1}{2^{8}}\sum_{i=1}^{3}\varepsilon(w_i)\,\D(w_i\mu)=0.
		\]
		
		\smallskip\noindent\textbf{Cases B--D: $c$ and $d$ not both odd.}
		By Lemma~\ref{lem:classmu}, $\mu$ lies in a nonzero class $X\in\{V,S,S'\}$. Choose $w_1$ to
		be a representative of the unique coset mapping $\mu$ into $V$. Then
		$\sigma(w_1\mu)=1$ and $\sigma(w_2\mu)=\sigma(w_3\mu)=-1$. By
		Proposition~\ref{prop:identity},
		\[
		\sum_{i=1}^{3}\varepsilon(w_i)\,\sigma(w_i\mu)\,\D(w_i\mu)
		=\varepsilon(w_1)\D(w_1\mu)-\varepsilon(w_2)\D(w_2\mu)-\varepsilon(w_3)\D(w_3\mu)
		=2\,\varepsilon(w_1)\D(w_1\mu).
		\]
		Proposition~\ref{prop:coset} therefore gives
		\begin{equation}\label{eq:theta-X}
			\Theta_\lambda(t)=-\frac{1}{2^{7}}\,\varepsilon(w_1)\,\D(w_1\mu).
		\end{equation}
		It remains to identify $w_1$ in each case.
		
		\smallskip\noindent\textbf{Case B: $c$ odd, $d$ even.} Here $\mu\in V$, so we may take
		$w_1=1$, and \eqref{eq:theta-X} gives $\Theta_\lambda(t)=-2^{-7}\D(\mu)$.
		
		\smallskip\noindent\textbf{Case C: $c$ even, $d$ odd.} Here $\mu\in S$. We have
		$\beta_1^\vee=(1,1,1,1)$ and $\langle\mu,\beta_1^\vee\rangle=\sum_k\mu_k$, which is odd
		because $\mu\in S$. Thus $s_{\beta_1}\mu=\mu-\bigl(\sum_k\mu_k\bigr)\beta_1$ is integral
		with coordinate sum $-\sum_k\mu_k$, which is odd. Hence $s_{\beta_1}\mu\in V$, so we may
		take $w_1=s_{\beta_1}$. Since $\varepsilon(s_{\beta_1})=-1$, \eqref{eq:theta-X} gives
		$\Theta_\lambda(t)=2^{-7}\D(s_{\beta_1}\mu)$.
		
		\smallskip\noindent\textbf{Case D: $c$ and $d$ both even.} Here $\mu\in S'$. We have
		$\beta_2^\vee=(1,1,1,-1)$ and $\langle\mu,\beta_2^\vee\rangle=\sum_k\mu_k-2\mu_4$. This is
		odd, since $\sum_k\mu_k$ is even and $2\mu_4$ is odd. Thus $s_{\beta_2}\mu$ is integral
		with coordinate sum $2\mu_4$, which is odd. Hence $s_{\beta_2}\mu\in V$, so we may take
		$w_1=s_{\beta_2}$, and \eqref{eq:theta-X} gives $\Theta_\lambda(t)=2^{-7}\D(s_{\beta_2}\mu)$.
		
		\smallskip\noindent\emph{The remaining assertions.}
		Since $\mu$ is regular, so is $w\mu$ for every $w\in W$; hence $\pi_H(w\mu)\neq0$ and
		$\D(w\mu)\neq0$. Therefore $\Theta_\lambda(t)\neq0$ in Cases B--D, which proves that
		$\Theta_\lambda(t)=0$ exactly when $c$ and $d$ are both odd. Finally, let $u\in W_H$ satisfy
		$uw\mu=(w\mu)^+$. Then $\D(w\mu)=\varepsilon(u)\D\bigl((w\mu)^+\bigr)$, and by the Weyl
		dimension formula for $B_4$,
		$\D\bigl((w\mu)^+\bigr)=\dim V^{\Spin(9,\C)}_{\nu}$ with $\nu=(w\mu)^+-\rho_H$. This gives
		$\Theta_\lambda(t)=\pm2^{-7}\dim V^{\Spin(9,\C)}_{\nu}$.
	\end{proof}

	\begin{remark}\label{rem:simplicity}
		The simplicity of Theorem~\ref{thm:main} comes from two facts. First, $W_H\backslash W$
		has only $3$ elements, so the coset formula (Proposition~\ref{prop:coset}) expresses
		$\Theta_\lambda(t)$ as a sum of three terms; for $\GL_n$ at an involution with centralizer
		$\GL_p\times\GL_q$ there are $\binom np$ terms. Second, the triality action of $W$ on
		$P/D_4$ forces the signs $\sigma(w_i\mu)$ to be either all equal or different on exactly
		one coset. Together with the vanishing identity (Proposition~\ref{prop:identity}), the
		three-term sum therefore reduces to a single term or to zero.
	\end{remark}
	
	\section{The principal involution}\label{sec:principal}
	
	We now turn to the other conjugacy class of involutions. For the principal involution
	$t_0=\rho^\vee(-1)$, the character values were determined by Nadimpalli, Pattanayak and
	Prasad; the following is the case $m=2$ of \cite[Theorem~4.1]{NPP}, restated in our
	notation.
	
	\begin{theorem}[{\cite[Theorem~4.1]{NPP}}]\label{thm:principal}
		Let $t_0\in F_4(\C)$ be the principal involution, whose centralizer is
		$\bigl(\SL(2,\C)\times\Sp(6,\C)\bigr)/\{\pm1\}$. Let
		$\lambda=a\omega_1+b\omega_2+c\omega_3+d\omega_4$ be a dominant weight, and let
		$\mu=\lambda+\rho$. Then $\Theta_\lambda(t_0)=0$ if and only if $c$ and $d$ are both odd.
		Otherwise the set
		\[
		\Psi_\mu=\{\alpha\in\Phi:\ \langle\mu,\alpha^\vee\rangle\in2\Z\}
		\]
		is a root system of type $A_1\times B_3$. With respect to the positive system
		$\Psi_\mu^+=\Psi_\mu\cap\Phi^+$, the weight $\nu=\tfrac12\mu-\rho_{\Psi_\mu}$ is dominant
		integral, and
		\[
		\Theta_\lambda(t_0)=(-1)^{a+b+c}\,2^{-1}\dim U_{\nu},
		\]
		where $\rho_{\Psi_\mu}$ is the half-sum of the roots in $\Psi_\mu^+$, and $U_\nu$ is the
		irreducible representation of $\SL(2,\C)\times\Spin(7,\C)$ of highest weight $\nu$.
		Explicitly, by the Weyl dimension formula,
		\[
		\dim U_\nu=\prod_{\alpha\in\Psi_\mu^+}
		\frac{\langle\mu,\alpha^\vee\rangle}{2\,\langle\rho_{\Psi_\mu},\alpha^\vee\rangle}.
		\]
	\end{theorem}
	
	Comparing Theorems~\ref{thm:main} and~\ref{thm:principal}, we obtain the following.
	
	\begin{corollary}\label{cor:all-involutions}
		Let $\lambda=a\omega_1+b\omega_2+c\omega_3+d\omega_4$ be a dominant weight of $F_4(\C)$.
		The following are equivalent:
		\begin{enumerate}
			\item $\Theta_\lambda(s)=0$ for some involution $s\in F_4(\C)$;
			\item $\Theta_\lambda(s)=0$ for every involution $s\in F_4(\C)$;
			\item $c$ and $d$ are both odd, that is, $\lambda+\rho\in D_4$.
		\end{enumerate}
	\end{corollary}
	
	\begin{proof}
		Every involution of $F_4(\C)$ is conjugate to $t$ or to $t_0$~\cite{BdS}, and
		$\Theta_\lambda$ is a class function. By Theorems~\ref{thm:main} and~\ref{thm:principal},
		$\Theta_\lambda$ vanishes at $t$, and also at $t_0$, exactly when $c$ and $d$ are both
		odd. Hence (1), (2) and (3) are equivalent. The reformulation $\lambda+\rho\in D_4$ in (3)
		is Lemma~\ref{lem:classmu}.
	\end{proof}
	
	\section{Concluding remarks}\label{sec:remarks}
	
	\begin{remark}\label{rem:contrast}
		The principal involution $t_0$ is conjugate to the lifts to $F_4(\C)$ of the regular
		elliptic element $w_0=-1$ of $W$, which is why the methods of~\cite{NPP} and~\cite{Re2}
		apply to it. The involution $t$ is not of this type; see~\cite[Ch.~8]{Kac}
		and~\cite{Re1} for the classification of torsion elements, and~\cite{Ca} for elliptic
		classes in Weyl groups. This is why a different argument is needed for $t$.
	\end{remark}
	
	\begin{remark}\label{rem:mechanism}
		The proof of Theorem~\ref{thm:main} uses two features of the pair
		$(F_4(\C),\Spin(9,\C))$: the index $|W:W_H|=3$ is small, and the sign
		$\nu\mapsto e^\nu(t)$ is constant on the classes of $P/D_4$, a group of order $4$ on which
		$W$ acts through $S_3$ with $W_H$ as the stabilizer of a nonzero class. It would be
		interesting to find analogous structures for the involutions of $E_6$, $E_7$ and $E_8$,
		for instance the involution of $E_8(\C)$ with centralizer $\Spin(16,\C)/\mu_2$.
	\end{remark}

\subsection*{Acknowledgements}
The author thanks Prof.~Dipendra Prasad for reading the manuscript and for his questions,
which led to Remark~\ref{rem:simplicity}. The mathematical results and proofs in this
paper are the author's own. An AI assistant was used to improve the presentation, to check
the main formulas numerically, and to search the literature. The author has verified all
content and takes full responsibility for it.

\end{document}